\documentclass{amsart}
\usepackage[all]{xy}
\usepackage{graphicx}

\newtheorem{theorem}{Theorem}[section]
\newtheorem{lemma}[theorem]{Lemma}
\newtheorem{proposition}[theorem]{Proposition}
\newtheorem{corollary}[theorem]{Corollary}

\theoremstyle{definition}
\newtheorem{definition}[theorem]{Definition}

\newtheorem{remark}[theorem]{Remark}

\renewcommand{\P}{{\mathbb{P}}}
\newcommand{\Z}{{\mathbb{Z}}}

\newcommand{\D}{{\mathbb{D}}}

\newcommand{\Ec}{{\mathcal{E}}}
\newcommand{\Fc}{{\mathcal{F}}}
\newcommand{\Gc}{{\mathcal{G}}}

\newcommand{\Kc}{{\mathcal{K}}}

\newcommand{\Oc}{{\mathcal{O}}}
\newcommand{\Qc}{{\mathcal{Q}}}
\newcommand{\Rc}{{\mathcal{R}}}
\newcommand{\Vc}{{\mathcal{V}}}
\newcommand{\Pc}{{\mathcal{P}}}

\newcommand{\Rb}{{\mathbf{R}}}
\newcommand{\Ub}{{\mathbf{U}}}
\newcommand{\Wb}{{\mathbf{W}}}

\newcommand{\ed}{\widehat{E}}

\newcommand\res[1]{{\lower1pt\hbox{$|$}}_{\raise.5pt\hbox{${\scriptstyle #1}$}}}

\numberwithin{equation}{section}

\begin{document}

\title{Tate Resolutions and Weyman Complexes}

\author{David A.\ Cox}
\address{Department of Mathematics and Computer Science, Amherst
College, Amherst, MA 01002-5000, USA}
\email{dac@cs.amherst.edu}

\author{Evgeny Materov}
\address{Zheleznogorsk's Branch of Sankt-Peterburg's University of EMERCOM, Severnaya~1, Zheleznogorsk, 662970 Russia} 
\email{materov@gmail.com}


\begin{abstract}
We construct generalized Weyman complexes for coherent sheaves on
projective space and describe explicitly how the differential depend
on the differentials in the correpsonding Tate resolution.  We apply
this to define the Weyman complex of a coherent sheaf on a projective
variety and explain how certain Weyman complexes can be regarded as
Fourier-Mukai transforms. 
\end{abstract}

\date{\today}

\maketitle

\section*{Introduction}

In this paper we study the relation between the terms and the maps of
two very important complexes in algebra and algebraic geometry:\ Tate
resolutions and Weyman complexes.  In the first section of the paper, 
we define the generalized Weyman complex of a coherent sheaf on
projective space and construct an explicit functor that takes the Tate
resolution of the sheaf to its generalized Weyman complex.
The second section then applies this to coherent sheaves on projective
varieties and Fourier-Mukai transforms.

We begin by recalling Tate resolutions and Weyman complexes.

\subsection*{Tate Resolutions} 
Let $K$ be a field of characteristic $0$ and $W$ a vector space over
$K$ of dimension $N+1$ with dual $W^*$.  Let $E = \bigwedge (W^*)$ be
the exterior algebra of $W^*$, $S = \mathrm{Sym}(W)$ the symmetric
algebra of $W$.  We grade $E$ and $S$ so that elements of $W$ have
degree $1$ and elements of $W^*$ have degree $-1$.  Thus
$\mathrm{Sym}^k(W)$ has the degree $k$ and $\bigwedge^k (W^*)$ has the
degree $-k$.  It is well-known that coherent sheaves on the projective
space $\P(W) = (W^*-\{0\})/\!\!\sim$ can be described in terms of
graded $S$-modules.  The
Bernstein-Gel$\!$\'{}$\!$fand-Gel$\!$\'{}$\!$fand (BGG) correspondence
described in \cite{BGG} (see also \cite[Chapter 4, \S~3
]{gelfand-manin}, \cite{kapranov_gras}) consists of a pair of adjoint
functors $\mathbf{R}$ and $\mathbf{L}$ which define an equivalence
between the derived category of bounded complexes of finitely
generated $S$-modules and the category of complexes of graded free
$E$-modules.  Eisenbud, Fl\o ystad and Schreyer showed in \cite{EFS}
that the essential part of the BGG correspondence is given via the
Tate resolution. Namely, a correspondence
\[
\mathbf{R}\,:\,
\{\text{sheaves on }\P(W)\}\to
\{\text{complexes of } E\text{-modules}\}.
\]
is given by assigning to a coherent sheaf $\Fc$ on $\P(W)$ a
bi-infinite complex (the Tate resolution)
\[
T^{\bullet}(\Fc) : \cdots \to 
T^p(\Fc)\to T^{p+1}(\Fc)
\to \cdots, \quad p\in\Z
\]
of free graded $E$-modules. By \cite[Theorem 4.1]{EFS} the terms of
the Tate resolution $T^\bullet(\Fc)$ are given by
\[
T^p(\Fc) = \bigoplus_{i}\widehat{E}(i-p)\otimes_K
H^i(\P(W),\Fc(p-i)),
\]
where $\widehat{E} = \mathrm{Hom}_K(E,K) = \omega_E$ is a dualizing
module for $E$.  Tate resolutions give constructive methods for
computing Beilison monads \cite{EFS} and Chow forms \cite{ES}, and are
also related to Castelnuovo-Mumford regularity---see \cite{CM2} for an
example.

The description of maps in Tate resolutions is a difficult and
challenging problem.  Knowledge of the maps leads to explicit formulas
for sparse resultants, as in \cite{khetan1, khetan2}.  Tate
resolutions for Veronese embeddings were described in \cite{Cox_bez},
where it was shown that the maps can be given by the Bezoutian of
homogeneous forms.  For Segre embeddings, the paper \cite{CM1} gives a
description of the maps via hyperdeterminants.

\subsection*{Weyman Complexes.}
We recall the Weyman complex
for a resultant.  Let $X$ be an irreducible projective variety with a
very ample line bundle $\Oc_X(1)$ and a vector bundle $\Vc$.  Let $W =
H^0(X,\Oc_X(1))$ and set $d = \dim(X)$.  The incidence variety
\[
Z = \{(f_0,\dots,f_d,x) \in W^{d+1}\times X \mid f_0(x) = \cdots =
f_d(x) = 0\}
\]
gives the commutative diagram of projections
\[
\xymatrix{
Z \ar[dr] \ar@/^/[drr]^{p_2} \ar@/_/[ddr]_{p_1} \\
& W^{d+1} \times X \ar[r]_(.62){\pi_2} \ar[d]^{\pi_1} & X\\
& W^{d+1}.
}
\]
Also let $A = \mathrm{Sym}((W^{d+1})^*)$ be the coordinate ring of the
affine space $W^{d+1}$.  Then Weyman's \emph{Basic Theorem for
Resultants} \cite[(9.1.2)]{weyman2} states that for the free graded
$A$-modules
\begin{equation}
\label{weymanFp}
F^p(\Vc) = \bigoplus_i {\textstyle\bigwedge^{i-p}} K^{d+1} \otimes_K
H^i(X,\Vc(p-i)) \otimes_K A(p-i),
\end{equation}
there exist minimal graded differentials of degree $0$
\[
F^p(\Vc) \longrightarrow F^{p+1}(\Vc)
\]
such that the resulting complex $F^\bullet(\Vc)$, when regarded as a
complex of sheaves on $W^{d+1}$, satisfies
\[
F^\bullet(\Vc) \simeq \Rb p_{1*}(p_2^*\Vc).
\]
The complex $F^\bullet(\Vc)$ has certain functorial properties
\cite[(9.1.2)]{weyman2} and can be used to compute resultants
\cite[(9.1.3)]{weyman2}.  Other properties of $F^\bullet(\Vc)$
are described in \cite[(5.1.3), (5.1.4), (5.1.6)]{weyman2}.  

Weyman writes \eqref{weymanFp} as $F_{-p}$ rather than $F^p$.  We use
$F^p$ to emphasize the relation to the Tate resolution.  The
relation becomes clear when we write the Weyman complex using the
projective embedding $i: X \hookrightarrow \P(W)$ induced by
$\Oc_X(1)$.  This gives the sheaf $\Fc = i_*\Vc$ on $\P(W)$, and
\eqref{weymanFp} becomes
\[
F^p(\Vc) = \bigoplus_i {\textstyle\bigwedge^{i-p}} K^{d+1} \otimes_K
H^i(\P(W),\Fc(p-i)) \otimes_K A(p-i).
\]
In contrast, the Tate resolution of $\Fc$ has
\begin{equation}
\label{tatetp}
T^p(\Fc) = \bigoplus_i \ed(i-p) \otimes_K
H^i(\P(W),\Fc(p-i)).
\end{equation}

\subsection*{Our Results}
Our main results can be summarized as follows:
\begin{itemize}
\item First, we give an explicit formula for the differentials in the
Weyman complex in terms of the differentials in the correponding Tate
resolution.
\item Second, our construction works for \emph{any} coherent sheaf
  $\Fc$ on $\P(W)$ and allows us to replace $W^{d+1}$ with $W^\ell$ for
  \emph{any} $\ell$ satisfying $1 \le \ell \le \dim(W) - 1$.
\end{itemize}
The idea is that for $1 \le \ell \le \dim(W) - 1$, there is an additive
functor $\mathbf{W}_\ell$ from the category of finitely generated
graded free $E$-modules and homomorphisms of degree $0$ such that
$\mathbf{W}_\ell^\bullet(\Fc) =
(\mathbf{W}_\ell(T^\bullet(\Fc)),\mathbf{W}_\ell(d^p))_{p\in\Z}$.

The statement and proof of our main result, Theorem~\ref{WlRF}, is the
subject of Section~1.  In Section~2 we define generalized Weyman
complexes for irreducible projective varieties and, in the case of a
vector bundle, interpret the Weyman complex as a Fourier-Mukai
transform.

\section{Generalized Weyman Complexes}

\subsection{Notation}
Fix an integer $\ell$ satisfying $1 \le \ell \le \dim(W) - 1$ and let
$A_\ell = \mathrm{Sym}((W^\ell)^*)$ be the coordinate ring of the
affine space $W^\ell$.  We give the polynomial ring $A_\ell$ the usual
grading where every variable has degree $1$.  Note also that $W^\ell =
K^\ell \otimes_K W$, so that $A_\ell = \mathrm{Sym}(K^\ell \otimes_K
W^*)$, provided we identify $K^\ell$ with its dual using the standard
basis of $K^\ell$.

\subsection{The Functor $\Wb_\ell$}
We define an additive functor $\Wb_\ell$ from finitely generated
graded free $E$-modules and homomorphisms of degree $0$ to finitely
generated graded free $A_\ell$-modules and homomorphisms of degree $0$
as follows.  For the free $E$-module $\ed(j)$, set
\[
\Wb_\ell(\ed(j)) = {\textstyle\bigwedge^j} K^\ell \otimes_K
A_\ell(-j).
\]
For morphisms, let $H$ and $H'$ be finite-dimensional $K$-vector
spaces and let
\[
\varphi : \ed(a)\otimes_K H \longrightarrow \ed(b)\otimes_K H' 
\]
be an $E$-module homomorphism of degree $0$.  By \cite[Prop.\
5.6]{EFS}, $\varphi$ is uniquely determined by
\[
\varphi_{-b} : \ed(a)_{-b}\otimes_K H = {\textstyle\bigwedge^{a-b}}W
\otimes_K H 
\longrightarrow \ed(b)_{-b}\otimes_K H' = H'.
\]
We define
\begin{equation}
\label{Wphitodefine}
\Wb_\ell(\varphi) : \Wb_\ell(\ed(a)) \longrightarrow \Wb_\ell(\ed(b))
\end{equation}
as follows.  Recall the comultiplication map
\[
\Phi : {\textstyle\bigwedge^a} K^\ell \longrightarrow
{\textstyle\bigwedge^b} K^\ell \otimes_K {\textstyle\bigwedge^{a-b}}
K^\ell
\]
and the injective map
\[
\Psi : {\textstyle\bigwedge^{a-b}} K^\ell \otimes_K
{\textstyle\bigwedge^{a-b}} W^* \longrightarrow
\mathrm{Sym}^{a-b}(K^\ell \otimes_K W^*) = (A_\ell)_{a-b}
\]
coming from the Schur functor decomposition of
$\mathrm{Sym}^{a-b}(K^\ell \otimes_K W^*)$ (see
\cite[(2.3.3)]{weyman2}).  Finally, $\varphi_{-b}$ determines
$\widetilde{\varphi}_{-b} : H \to \bigwedge^{a-b} W^* \otimes H'$.
Putting these maps together, we get the composition
\begin{equation}
\label{Wlphi}
\begin{array}{c}
\xymatrix@C=35pt{
{\textstyle\bigwedge^a} K^\ell \otimes_K H \ar[r]^(.3){\Phi\otimes
  \widetilde{\varphi}_{-b}} & 
{\textstyle\bigwedge^b} K^\ell \otimes_K
{\textstyle\bigwedge^{a-b}} K^\ell \otimes_K
{\textstyle\bigwedge^{a-b}} W^* \otimes_K H' \ar[d]^{1\otimes 
\Psi \otimes 1}\\ 
& {\textstyle\bigwedge^b} K^\ell \otimes_K 
(A_\ell)_{a-b} \otimes_K H',
}\end{array}
\end{equation}
which in turn determines an $A_\ell$-module homomorphism of degree $0$
\[
\Wb_\ell(\varphi) : {\textstyle\bigwedge^a} K^\ell \otimes_K H
\otimes_K A_\ell(-a) \longrightarrow 
{\textstyle\bigwedge^b} K^\ell \otimes_K H' \otimes_K A_\ell(-b).
\]
This is the desired map \eqref{Wphitodefine}.  Note that $\Psi$ is
given explicitly by
\[
v_1\wedge \cdots \wedge v_{a-b} \otimes w_1\wedge \cdots \wedge
w_{a-b} \mapsto \sum_{\sigma \in S_{a-b}}
\mathrm{sgn}(\sigma)\,v_1\otimes w_{\sigma(1)} \cdots v_{a-b}\otimes
w_{\sigma(a-b)},
\]
and $\Phi$ is given by the following formula
\[
v_1\wedge\cdots\wedge v_a\mapsto
\sum \mathrm{sgn}(I, I')\, v_I\otimes v_{I'},
\]
where $v_I = v_{i_1}\wedge \cdots \wedge v_{i_p}$, $I =
\{i_1,\ldots,i_p\}\subset \{1,\ldots,a\}$, $\mathrm{sgn}(I, I')$ is
the sign of the permutation whose bottom row is $(I, I')$, and $I' =
\{1,\ldots,a\}\backslash I$ is the complement to $I$.  It follows that
we have a completely explicit formula for $\Wb_\ell(\varphi)$.

\subsection{Generalized Weyman Complexes}
We now define the generalized Weyman complex of a coherent sheaf on
$\P(W)$.

\begin{definition}
\label{weymandef}
Let $T^\bullet(\Fc) = (T^p,d^p)_{p \in \Z}$ be the Tate resolution of
a coherent sheaf $\Fc$ on $\P(W)$.  Then the \emph{\bfseries $\ell$-th
Weyman complex} of $\Fc$ is the complex of free graded
$A_\ell$-modules given by
\[
\Wb_\ell^\bullet(\Fc) = \big(\Wb_\ell(T^p),\Wb_\ell(d^p)\big)_{p \in \Z}.
\]
\end{definition}

Using \eqref{tatetp}, it follows that
\[
\Wb_\ell^p(\Fc) = \bigoplus_i {\textstyle\bigwedge^{i-p}} K^\ell
\otimes_K H^i(\P(W),\Fc(p-i)) \otimes_K A_\ell(p-i).
\]
Note also that $\Wb_\ell^\bullet(\Fc)$ is a complex by functorality
and is minimal since the Tate resolution $T^\bullet(\Fc)$ is minimal.
In contrast to the bi-infinite Tate resolution, however, the Weyman
complex is finite.  More precisely, one sees easily that
\[
\Wb_\ell^p(\Fc) = 0\quad \text{whenever}\ \ p < -\ell \ \ \text{or}\ \
p > \dim(\mathrm{supp}(\Fc)). 
\]

\begin{remark}
\label{Ulremark}
Our results about $\Wb_\ell^\bullet(\Fc)$ are parallel (and were inspired
by) the properties of the functor $\Ub_\ell^\bullet$ introduced by
Eisenbud and Schreyer in \cite{ES}.  They define this functor
on graded free $E$-modules by
\[
\Ub_\ell(\ed(j)) = {\textstyle\bigwedge^j}U_\ell,
\]
where $U_\ell$ is the tautological subbundle on the Grassmannian
$\mathbb{G}_\ell$ of codimension $\ell$ subspaces of $\P(W)$.  When
applied to the Tate resolution of a coherent sheaf $\Fc$ on $\P(W)$,
this gives the complex $\Ub_\ell^\bullet(\Fc)$ of locally free sheaves
on $\mathbb{G}_\ell$ defined in \cite{ES}.  As we develop the
properties of $\Wb_\ell^\bullet$, we will explain how our functor
relates to the corresponding properties of $\Ub_\ell^\bullet$.
\end{remark}

Our first result, similar to \cite[Prop.\ 1.3]{ES}, shows that $\Fc$
is uniquely determined by $\Wb_\ell^\bullet(\Fc)$ provided $\ell$ is
sufficiently large.

\begin{proposition}
\label{WldeterminesF}
Let $\Fc$ be a coherent sheaf on $\P(W)$.  If $\ell >
\dim(\mathrm{supp}(\Fc))$, then $\Fc$ is determined by the complex
$\Wb_\ell^\bullet(\Fc)$. 
\end{proposition}

\begin{proof}
A map $\varphi : \ed(a) \otimes_K H \to \ed(b) \otimes_K
H'$ gives $\Wb_\ell(\varphi) : \Wb_\ell(\ed(a)) \to \Wb_\ell(\ed(b))$.
We claim that $\varphi \mapsto \Wb_\ell(\varphi)$ is injective when $0
\le b \le a \le \ell$.

To prove this, assume that $\Wb_\ell(\varphi)$ is trivial.  This
implies that the map defined in \eqref{Wlphi} vanishes.  Since $\Psi$
is injective, it follows that
\[
\Phi \otimes \widetilde{\varphi}_{-b}: 
{\textstyle\bigwedge^a} K^\ell \otimes_K H \to 
{\textstyle\bigwedge^b} K^\ell \otimes_K {\textstyle\bigwedge^{a-b}}
K^\ell \otimes_K {\textstyle\bigwedge^{a-b}} W^* \otimes_K H'
\]
is trivial, which in turn implies that
\[
\Phi \otimes \varphi_{-b}: 
{\textstyle\bigwedge^a} K^\ell \otimes_K
{\textstyle\bigwedge^{a-b}} W \otimes_K H \to {\textstyle\bigwedge^b}
K^\ell\otimes_K {\textstyle\bigwedge^{a-b}} K^\ell \otimes_K H'
\]
is trivial.  Then the commutative diagram
\[
\xymatrix@C=35pt{
{\textstyle\bigwedge^a} K^\ell \otimes_K 
{\textstyle\bigwedge^{a-b}} W \otimes_K 
H \ar[r]^(.42){\Phi\otimes 1 \otimes 1} \ar[d]^{1\otimes\varphi_{-b}}& 
{\textstyle\bigwedge^b} K^\ell \otimes_K {\textstyle\bigwedge^{a-b}}
K^\ell \otimes_K {\textstyle\bigwedge^{a-b}} W \otimes_K 
H \ar[d]^{1\otimes 1\otimes
\varphi_{-b}}\\ 
{\textstyle\bigwedge^a} K^\ell \otimes_K H' \ar[r]^(.43){\Phi\otimes 1}
& {\textstyle\bigwedge^b} K^\ell \otimes_K
{\textstyle\bigwedge^{a-b}} K^\ell \otimes_K H'
}
\]
and the injectivity of $\Phi$ (this uses $0 \le b \le a \le \ell$)
show that $1\otimes \varphi_{-b} = 0$.  It follows that
$\varphi_{-b}$ and hence $\varphi$ are trivial, as claimed.

Now we argue as in the proof of \cite[Prop.\ 1.3]{ES}.  Set $d =
\dim(\mathrm{supp}(\Fc))$, so that $T^{-1}(\Fc) \to T^0(\Fc)$ looks
like 
\[
\xymatrix@R=0pt@C=15pt{
\ed(d+1) \otimes_K H^d(\P(W),\Fc(-d-1)) \ar[r] \ar[ddr] \ar[ddddr] &
\ed(d) \otimes_K H^d(\P(W),\Fc(-d))\\ 
\bigoplus & \bigoplus \\ 
 \vdots \ar[ddr]& \vdots \\ 
\bigoplus & \bigoplus \\ 
\ed(1) \otimes_K H^0(\P(W),\Fc(-1)) \ar[r]  &
\ed(0) \otimes_K H^0(\P(W),\Fc(0)).
}
\]
Since $\ell > d = \dim(\mathrm{supp}(\Fc))$, all of the maps $\varphi$
in this diagram satisfy $0 \le a \le b \le \ell$, which means that
these maps are determined uniquely by the map $\Wb_\ell^{-1}(\Fc) \to
\Wb_\ell^{0}(\Fc)$ in the Weyman complex.  Thus the Weyman complex of
$\Fc$ determines $T^{-1}(\Fc) \to T^0(\Fc)$, which in turn determines
$\Fc$ by usual properties of the Tate resolution (see the proof of
\cite[Prop.\ 1.3]{ES} for the details).
\end{proof}

\subsection{Main Theorem}
Now consider the incidence variety
\[
Z_\ell = \{(f_1,\dots,f_\ell,x) \in W^\ell\times \P(W) \mid f_1(x) =
\cdots = f_\ell(x) = 0\}
\]
and the commutative diagram of projections
\begin{equation}
\label{Zldiagram}
\begin{array}{c}
\xymatrix{
Z_\ell \ar[dr] \ar@/^/[drr]^{p_2} \ar@/_/[ddr]_{p_1} \\
& W^{\ell} \times \P(W) \ar[r]_(.6){\pi_2}
\ar[d]^{\pi_1} & \P(W)\\ 
& W^{\ell}.
}\end{array}
\end{equation}

Here is the main theorem of this section.

\begin{theorem}
\label{WlRF}
Let $\Fc$ be a coherent sheaf on $\P(W)$ and assume $1 \le \ell \le
\dim(W)-1$.  Then $\Wb_\ell^\bullet(\Fc)$, when regarded as a complex
of sheaves on $W^\ell$, represents $\mathbf{R}p_{1*}(p_2^*\Fc)$ in the
derived category of coherent sheaves on $W^\ell$.
\end{theorem}

Before we can prove this theorem, however, we need some preliminary
work.

\subsection{A Special Weyman Complex}
First observe that since $W^\ell = \mathrm{Spec}(A_\ell)$ is affine,
$\mathbf{R}p_{1*}(p_2^*\Fc)$ can be regarded as a complex of sheaves
that computes the cohomology groups $H^q(Z_\ell,p_2^*\Fc)$.  We also
note that the natural action of $K^\times = K\setminus\{0\}$ on
$W^\ell$ (and corresponding trivial action on $\P(W)$) makes
$H^q(Z_\ell,p_2^*\Fc)$ into a graded $A_\ell$-module.  Here is a
special case where this module is known explicitly.

\begin{proposition}
\label{WOaa}
If $\ell$ satisfies $1 \le \ell \le \dim(W)-1$, then the graded
$A_\ell$-module $H^q(Z_\ell,p_2^*\Omega^a_{\P(W)}(a))$ is given by
\[
H^q(Z_\ell,p_2^*\Omega^a_{\P(W)}(a)) = \begin{cases}
{\textstyle\bigwedge^{a}} K^\ell \otimes_K A_\ell(-a) =
\Wb_\ell(\ed(a)) & q = 0 \\ 0 & q > 0. \end{cases}
\]
\end{proposition}

\begin{proof}
The proof of \cite[(5.1.2)]{weyman2} is easily adapted to show
that $H^q(Z_\ell,p_2^*\Omega^a_{\P(W)}(a))$ is the $q$-th cohomology of
the complex 
\begin{align}
\label{wey_complex}
\cdots \to\ &\bigoplus_{i\ge 0} {\textstyle\bigwedge^{i-q}} K^\ell
\otimes_K H^i(\P(W),\Omega^a_{\P(W)}(a +q-i)) \otimes_K A_\ell(q-i)
\to \nonumber \\ &\bigoplus_{i\ge0} {\textstyle\bigwedge^{i-q-1}} K^\ell
\otimes_K H^i(\P(W),\Omega^a_{\P(W)}(a+q+1-i)) \otimes_K A_\ell(q+1-i)
\to \cdots
\end{align}

Note $\dim(\P(W)) = N$ since $\dim(W) = N+1$.  The Bott
vanishing theorem (proved in many papers, including \cite{Huang})
states that
\[
H^i(\P(W),\Omega^a_{\P(W)}(k)) = 0
\]
whenever
\[
i \notin \{0,a,N\}\ \text{ or }\ i = 0, k \le a\ \text{ or }\ i = a, k
\ne 0 \ \text{ or }\ i = N, k \ge a-N.
\]
Then the general $q$-th term of the above complex for $i \in \{0, a,
N\}$ reduces to
\begin{align*}
&{\textstyle\bigwedge^{-q}} K^\ell \otimes_K
H^0(\P(W),\Omega^a_{\P(W)}(a+q)) \otimes_K A_\ell(q)
\\ {\displaystyle\bigoplus}\ 
&{\textstyle\bigwedge^{a-q}} K^\ell \otimes_K
H^a(\P(W),\Omega^a_{\P(W)}(q)) \otimes_K A_\ell(q-a)
\\ {\displaystyle\bigoplus}\ 
&{\textstyle\bigwedge^{N-q}} K^\ell \otimes_K
H^N(\P(W),\Omega^a_{\P(W)}(a+q-N)) \otimes_K A_\ell(q-N).
\end{align*}
There are the following cases with respect to the different $i$:
\begin{itemize}
\item If $i = 0$, then either ${\textstyle\bigwedge^{-q}} K^{\ell}= 0$
for $q>0$, or $H^0(\P(W),\Omega^a_{\P(W)}(a+q)) = 0$ for $q\le 0$ from
the Bott vanishing theorem. Therefore, the first summand is zero for
all $q$.
\item If $i = a$, then the Bott vanishing theorem implies 
\begin{align*}
&\quad{\textstyle\bigwedge^{a-q}} K^\ell \otimes_K
H^a(\P(W),\Omega^a_{\P(W)}(q)) \otimes_K A_\ell(q-a)
\\ =\ & \begin{cases}
{\textstyle\bigwedge^{a}} K^\ell \otimes_K
H^a(\P(W),\Omega^a_{\P(W)}) \otimes_K A_\ell(-a) &
q = 0 \\ 0 & q \ne 0. \end{cases}
\end{align*}
\item If $i = N$, then ${\textstyle\bigwedge^{N-q}} K^\ell = 0$ for
$q<0$ because we assume that $\ell \le \dim(W) - 1 = N$. If $q\ge 0$,
then the Bott vanishing theorem says that
$H^N(\P(W),\Omega^a_{\P(W)}(a + q-N)) = 0$.
\end{itemize}
Our conclusion is that the complex reduces to 
\[
\cdots \longrightarrow {\textstyle\bigwedge^{a}} K^\ell
\otimes_K H^a(\P(W),\Omega^a_{\P(W)}) \otimes_K A_\ell(-a)
\longrightarrow \cdots,
\]
where the nonzero module occurs where $q = 0$.  Since
$H^a(\P(W),\Omega_{\P(W)}^a) = K$, it follows that
\[
H^q(Z_\ell,p_2^*\Omega_{\P(W)}^a(a)) = \begin{cases}
{\textstyle\bigwedge^{a}} K^\ell \otimes_K A_\ell(-a) =
\Wb_\ell(\ed(a)) & q = 0 \\ 0 & q > 0. \end{cases}\qedhere
\]
\end{proof}

\begin{corollary}
\label{propcor}
The derived functor
$\mathbf{R}p_{1*}(p_2^*\Omega_{\P(W)}^a(a))$ is canonically
represented by the complex of graded $A_\ell$-modules consisting of
$\Wb_\ell(\ed(a))$ concentrated in degree
$0$.  In other words,
\[
\mathbf{R}p_{1*}(p_2^*\Omega_{\P(W)}^a(a)) =
p_{1*}(p_2^*\Omega_{\P(W)}^a(a)) = \Wb_\ell(\ed(a)).
\]
\end{corollary}

\subsection{Homomorphisms} Proposition~\ref{WOaa} gives isomorphisms
\[
\gamma^a : \Wb_\ell(\ed(a)) \simeq
H^0(Z_\ell,p_2^*\Omega_{\P(W)}^a(a)),\quad a \ge 0.
\]
We next explore how these isomorphisms interact with the homomorphisms
$\Wb_\ell(\varphi)$ from \eqref{Wphitodefine}.  The simplest case
begins with an element
\[
\phi \in {\textstyle\bigwedge^{a-b}} W^*,
\]
which gives two maps as follows.  To construct the first map, note
that $\phi$ induces a map $\varphi_{-b} : \bigwedge^{a-b} W \to K$,
which gives an $E$-module homomorphism of degree 0
\[
\varphi: \ed(a) \longrightarrow \ed(b).
\]
Applying \eqref{Wphitodefine}, we obtain a graded $A_\ell$-module
homomorphism
\[
\Wb_\ell(\varphi) : \Wb_\ell(\ed(a)) \longrightarrow \Wb_\ell(\ed(b)).
\] 
This is the first map.  For the second map, we recall the standard
exact sequence
\begin{equation}
\label{PWses}
0 \longrightarrow \Omega_{\P(W)}^1(1) \longrightarrow W \otimes
\Oc_{\P(W)} \longrightarrow \Oc_{\P(W)}(1) \longrightarrow 0.
\end{equation}
From contraction with $\phi$ we get a sheaf morphism
\[
\widetilde{\phi}: \Omega^a_{\P(W)}(a) \longrightarrow
\Omega^b_{\P(W)}(b).
\]
This gives a graded $A_\ell$-module homomorphism
\[
\widehat{\phi} : H^0(Z_\ell,p_2^*\Omega^a_{\P(W)}(a)) \longrightarrow
H^0(Z_\ell,p_2^*\Omega^b_{\P(W)}(b)),
\]
which is the second map.  These maps are related as follows.

\begin{proposition}
\label{homomorphism}
The above maps fit into a commutative diagram:
\[
\xymatrix{
H^0(Z_\ell,p_2^*\Omega^a_{\P(W)}(a)) \ar[r]^{\widehat{\phi}} &
H^0(Z_\ell,p_2^*\Omega^b_{\P(W)}(b))\\
\Wb_\ell(\ed(a)) \ar[u]^{\gamma^a} \ar[r]^{\Wb_\ell(\varphi)} & 
\Wb_\ell(\ed(b)) \ar[u]_{\gamma^b}.
}
\]
\end{proposition}

Proving this will require an intrinsic description of the graded
structure of $H^0(Z_\ell,p_2^*\Omega_{\P(W)}^a(a))$ in terms of Schur
functors.

\subsection{Graded Structures and Schur Functors}
For simplicity, we write the exact sequence \eqref{PWses} as
\[
0 \longrightarrow \Rc \longrightarrow W \otimes
\Oc_{\P(W)} \longrightarrow \Qc \longrightarrow 0.
\]
For simplicity, we will use $\otimes$ denote both
$\otimes_{\Oc_{\P(W)}}$ (for sheaves) and $\otimes_K$ (for vector
spaces) in the remainder of the paper.  The meaning of $\otimes$
should be clear from the context.

One easily sees that $Z_\ell \subseteq W^\ell \times \P(W)$ is the
total space of the vector bundle $\Rc^\ell = K^\ell \otimes \Rc$, and
the direct image $p_{2*}(\Oc_{Z_\ell})$ can be identified with the
sheaves of algebras
$\mathrm{Sym}((\Rc^\ell)^*)\cong\mathrm{Sym}(K^\ell \otimes \Rc^*)$
(see \cite[Proposition 5.1.1(b)]{weyman2}).  Then the projection formula
yields that
\begin{align*}
H^0(Z_\ell,p_2^*\Omega_{\P(W)}^a(a)) &\simeq
H^0(\P(W),p_{2*}(p_2^*\Omega_{\P(W)}^a(a)))\\ &=
H^0(\P(W),\mathrm{Sym}(K^\ell \otimes \Rc^*) \otimes
\Omega_{\P(W)}^a(a))\\ &= \bigoplus_{k\ge0}
H^0(\P(W),\mathrm{Sym}^k(K^\ell \otimes \Rc^*) \otimes
\Omega_{\P(W)}^a(a)).
\end{align*}
We identify $K^\ell$ with its dual using the standard basis.  Hence 
$H^0(Z_\ell,p_2^*\Omega_{\P(W)}^a(a))$ is the graded $A_\ell$-module
whose graded piece in degree $k$ is
\begin{equation}
\label{h0tostudy}
H^0(Z_\ell,p_2^*\Omega_{\P(W)}^a(a))_k =
H^0(\P(W),\mathrm{Sym}^k(K^\ell \otimes \Rc^*)\otimes 
\Omega_{\P(W)}^a(a)).
\end{equation}

Note that $\Rc = \Omega_{\P(W)}^1(1)$ has rank $N$ since $\dim(W) =
N+1$.  Combining this with $\Qc = \Oc_{\P(W)}(1)$, we obtain
\[
\Omega_{\P(W)}^a(a) = {\textstyle\bigwedge^a}\Rc \simeq  
{\textstyle\bigwedge^{N-a}}\Rc^* \otimes {\textstyle\bigwedge^N}\Rc
 \simeq {\textstyle\bigwedge^{N-a}}\Rc^* \otimes \Qc^*,
\]
where the second equality is standard duality and the third follows
from
\[
{\textstyle\bigwedge^N}\Rc = \Omega_{\P(W)}^N(N) = \Oc_{\P(W)}(-N-1 +
N) = \Oc_{\P(W)}(-1) =  \Qc^*.
\]
Hence graded piece \eqref{h0tostudy} is
\begin{equation}
\label{h0tostudy2}
H^0(Z_\ell,p_2^*\Omega_{\P(W)}^a(a))_k =
H^0(\P(W),\mathrm{Sym}^k(K^\ell \otimes \Rc^*)\otimes
{\textstyle\bigwedge^{N-a}}\Rc^* \otimes \Qc^*).
\end{equation}

We use Schur functors to decompose \eqref{h0tostudy2}, based on
\cite[Chap.\ 2]{weyman2}.  In Weyman's notation of 
\cite{weyman2}, the Schur functor of
a partition $\lambda$ is denoted $L_\lambda$.  This is not consistent
with \cite{FH}, which parametrizes Schur functors using conjugate
partitions.  Weyman's treatment also uses Weyl functors, which he
denotes $K_\lambda$.  Fortunately, since we are in characteristic $0$,
there is a canonical isomorphism $K_\lambda \simeq L_{\lambda'}$,
where $\lambda'$ is the conjugate partition of $\lambda$.

We will also need the following consequence of Weyman's version of
the Bott Theorem \cite[Chap.\ 4]{weyman2}.

\begin{lemma}
\label{Bottlemma}
Given a partition $\pi$, we have
\[
H^q(\P(W),L_\pi(\Rc^*)\otimes \Qc^*) = 0, \quad q > 0
\]
and 
\[
H^0(\P(W),L_\pi(\Rc^*)\otimes \Qc^*) = 
\begin{cases} L_\mu(W^*) & \pi = (N,\mu),\ \mu = (a_2,\dots,a_s)\\
0 & \text{\rm otherwise}.
\end{cases}
\]
\end{lemma}

\begin{proof}
Write $\pi = (a_1,\dots,a_s)$.  If $a_1 > N = \mathrm{rank}(\Rc^*)$,
then $L_\pi(\Rc^*) = 0$ and we are done.  Hence we may assume
$a_1 \le N$.  This allows us to write the conjugate partition as
\[
\pi' = (b_1,\dots,b_N),\quad b_1 \ge \cdots \ge b_N \ge 0.
\]
If we set $\alpha = (b_1,\dots,b_N,1)$, then Weyman defines \cite[p.\
115]{weyman2} the vector bundle
\[
\Vc(\alpha) = K_{\pi'}(\Rc^*)\otimes \Qc^*. 
\] 
Using the relation between Weyl and Schur functors noted above, we
write this as 
\[
\Vc(\alpha)  = L_\pi(\Rc^*)\otimes \Qc^*
\]
There are now two cases to consider.

\medskip

\noindent {\bf Case 1:} If $a_1 < N$, then $b_N = 0$, so that $\alpha
= (b_1,\dots,b_{N-1},0,1)$.  Then part (1) of \cite[(4.1.8)]{weyman2}
implies that $H^q(\P(W),L_\pi(\Rc^*)\otimes \Qc^*) = 0$ for all $q \ge
0$.  See also the discussion of item (1) of \cite[(4.1.5)]{weyman2}.

\medskip

\noindent {\bf Case 2:} If $a_1 = N$, then $b_N \ge 1$, so that
$\alpha$ is nonincreasing.  Here, part (2) of \cite[(4.1.8) and
(4.1.4)]{weyman2} imply that $H^q(\P(W),L_\pi(\Rc^*)\otimes \Qc^*) =
0$ for $q > 0$ and that
\[
H^0(\P(W),L_\pi(\Rc^*)\otimes \Qc^*) = K_{(b_1-1,\dots,b_N-1)}(W^*)
  \otimes ({\textstyle\bigwedge^{N+1}}W)^{\otimes(-1)},
\]
One easily sees that $(b_1-1,\dots,b_N-1)' = (a_2,\dots,a_s) = \mu$.
Hence
\[
K_{(b_1-1,\dots,b_N-1)} = L_{(b_1-1,\dots,b_N-1)'} = L_\mu,
\]
and the lemma follows since $\bigwedge^{N+1}W \simeq K$.
\end{proof}

We now decompose \eqref{h0tostudy2} by writing $\mathrm{Sym}^k(K^\ell
\otimes \Rc^*)\otimes {\textstyle\bigwedge^{N-a}}\Rc^*$ in terms of
Schur functors and then apply Lemma~\ref{Bottlemma}.  We begin with
the Cauchy formula
\[
\mathrm{Sym}^k(K^\ell \otimes \Rc^*) = \bigoplus_{|\lambda| = k}
L_\lambda(K^\ell) \otimes L_\lambda(\Rc^*),
\]
where the direct sum is over all partitions $\lambda$ of $k$ (see
\cite[(2.3.3)]{weyman2}).  Then the Pieri formula gives
\[
L_\lambda(\Rc^*) \otimes {\textstyle\bigwedge^{N-a}}\Rc^* = 
\bigoplus_{\pi \in Y(\lambda,N-a)} L_\pi(\Rc^*),
\]
where the direct sum is over all partitions $\pi \in Y(\lambda,N-a)$
whose Young diagram is obtained by adding $N-a$ boxes to the Young
diagram of $\lambda$, with no two boxes in the same column (see
\cite[(2.3.5)]{weyman2}). Combining these, we have
\begin{equation}
\label{symkformula}
\begin{aligned}
&\  \mathrm{Sym}^k(K^\ell \otimes \Rc^*) \otimes
  {\textstyle\bigwedge^{N-a}}\Rc^* \otimes \Qc^*\\ 
=\ &\bigoplus_{|\lambda| = k} \Big(\bigoplus_{\pi \in Y(\lambda,N-a)}
  L_\lambda(K^\ell) \otimes L_\pi(\Rc^*) \otimes \Qc^*\Big). 
\end{aligned}
\end{equation}
Taking global sections and applying Lemma~\ref{Bottlemma}, we get a
Schur functor decomposition of the graded piece \eqref{h0tostudy2}.

When we do this, Lemma~\ref{Bottlemma} tells us that we need only
consider partitions $\lambda$ of $k$ such that adding $N-a$ boxes to a
partition $\lambda$ gives a partition $\pi$ of the form $\pi =
(N,a_2,\dots,a_s)$.  This has some nice consequences for the graded
pieces \eqref{h0tostudy2}:
\begin{itemize}
\item If $k < a$, then adding $N-a$ boxes to $\lambda$ gives a
partition $\pi$ of $k+N-a < N$, so that $\pi$ is never of the form
$\pi = (N,a_2,\dots,a_s)$.  Then \eqref{symkformula} and
Lemma~\ref{Bottlemma} easily imply that the graded piece
\eqref{h0tostudy2} in degree $k < a$ vanishes, i.e.,
\[
H^0(Z_\ell,p_2^*\Omega_{\P(W)}^a(a))_k = 0.
\]
\item If $k = a$, then adding $N-a$ boxes to $\lambda$ gives a
partition $\pi$ of $a+N-a = N$, so that $\pi$ is of the form
$\pi = (N,a_2,\dots,a_s)$ if and only if $\lambda = (a)$ and $\pi =
(N)$.  Then \eqref{symkformula} and
Lemma~\ref{Bottlemma} imply that the graded piece in degree $a$ is
\begin{align*}
H^0(Z_\ell,p_2^*\Omega_{\P(W)}^a(a))_a &= H^0(\P(W),L_{(a)}(K^\ell)
\otimes L_{(N)}(\Rc^*) \otimes \Qc^*)\\
&= L_{(a)}(K^\ell) \otimes H^0(\P(W), L_{(N)}(\Rc^*)
\otimes \Qc^*)\\
&= {\textstyle\bigwedge^a}K^\ell \otimes L_{(0)}(W^*) =
	{\textstyle\bigwedge^a}K^\ell,
\end{align*}
since $L_{(a)}(K^\ell) = \bigwedge^a K^\ell$ and $L_{(0)}(W^*) =
\bigwedge^0 W^* = K$.
\end{itemize}
These bullets explain the shift by $a$ in the isomorphism
\begin{align*}
H^0(Z_\ell,p_2^*\Omega_{\P(W)}^a(a)) &\simeq \bigoplus_{k\ge0} 
H^0(\P(W),\mathrm{Sym}^k(K^\ell \otimes \Rc^*)\otimes
{\textstyle\bigwedge^{N-a}}\Rc^* \otimes \Qc^*)\\
&\simeq
{\textstyle\bigwedge^{a}}K^\ell \otimes A_\ell(-a) = \Wb_\ell(\ed(a))
\end{align*}
from Proposition~\ref{WOaa}.  

\subsection{Proof of Proposition~\ref{homomorphism}} For $\gamma^a :
\Wb_\ell(\ed(a)) \simeq H^0(Z_\ell,p_2^*\Omega_{\P(W)}^a(a))$, the
above analysis shows that the degree $k$ component of $\gamma^a$ is an
isomorphism
\[
\gamma_k^a : {\textstyle\bigwedge^{a}}K^\ell \otimes (A_\ell)_{k-a}
\simeq H^0(\P(W),\mathrm{Sym}^k(K^\ell \otimes \Rc^*)\otimes
{\textstyle\bigwedge^{N-a}}\Rc^* \otimes \Qc^*).
\]
Furthermore, Lemma~\ref{Bottlemma} and \eqref{symkformula} give a
Schur functor decomposition of the right-hand side, and the Cauchy and
Pieri formulas also give a Schur functor decomposition of the
left-hand side since $(A_\ell)_{k-a} =
\mathrm{Sym}^{k-a}(K^\ell\otimes W^*)$.  The isomorphisms $\gamma_k^a$
are clearly compatible with these decompositions.  Using this, we can
now ready for the proof.

\begin{proof}[Proof of Proposition~\ref{homomorphism}]
Since $\Wb_\ell(\ed(a))$ is generated in degree $a$, it suffices to
show that the diagram commutes in degree $a$.  Using
\eqref{h0tostudy2}, the map $\widetilde{\varphi}_{-b}:K\to
\bigwedge^{a-b} W^*$ induced by $\phi$, and the maps $\Phi$ and $\Psi$
from diagram \eqref{Wlphi}, we can write this as
\[
\xymatrix@C=10pt{
H^0(\mathrm{Sym}^a(K^\ell \otimes \Rc^*)\otimes
{\textstyle\bigwedge^{N-a}}\Rc^* \otimes \Qc^*) \ar[r]^{\widehat{\phi}} &
H^0(\mathrm{Sym}^a(K^\ell \otimes \Rc^*)\otimes
{\textstyle\bigwedge^{N-b}}\Rc^* \otimes \Qc^*)\\
& {\textstyle\bigwedge^{b}} K^\ell \otimes \mathrm{Sym}^{a-b}(K^\ell
\otimes W^*) \ar[u]_{\gamma^b_a} \\ 
{\textstyle\bigwedge^{a}} K^\ell \ar[r]_{\Phi \otimes
  \widetilde{\varphi}_{-b}} 
\ar[uu]^{\gamma_a^a}\ar[ur]^{\Wb_\ell(\varphi)} 
& {\textstyle\bigwedge^{b}} K^\ell \otimes {\textstyle\bigwedge^{a-b}}
K^\ell \otimes {\textstyle\bigwedge^{a-b}} W^*. \ar[u]_{1\otimes\Psi} 
}
\]
We have simplified notation by omitting $\P(W)$ when we
write sheaf cohomology.

To prove commutativity, first recall that in the Schur functor
decomposition of the cohomology group on the top left, the only
partition $\lambda$ that appears is $\lambda = (a)$ and that the top
left cohomology group can be replaced with 
\[
H^0(\mathrm{Sym}^a(K^\ell \otimes \Rc^*)\otimes
{\textstyle\bigwedge^{N-a}}\Rc^* \otimes \Qc^*) =
{\textstyle\bigwedge^{a}} K^\ell.
\]  
This maps to the $\lambda = (a)$ part of the cohomology group on the
top right, where the only element $\pi \in Y(\lambda,N-b)$ of the form
$\pi = (N,a_2,\dots,a_s)$ is $\pi = (N,a-b)$.  Hence the top right 
cohomology can be replaced with
\[
H^0({\textstyle\bigwedge^{a}} K^\ell \otimes L_{(N,a-b)}(\Rc^*)\otimes
\Qc^*) = {\textstyle\bigwedge^{a}} K^\ell \otimes L_{(a-b)}(W^*) =
{\textstyle\bigwedge^{a}} K^\ell\otimes {\textstyle\bigwedge^{a-b}}W^*.  
\]
The map between these cohomology groups is clearly $1\otimes
\widetilde{\varphi}_{-b}$.  Furthermore, one also sees that the only
way the partition $\lambda = (a)$ occurs in the Schur decompostion of
\[
{\textstyle\bigwedge^{b}} K^\ell \otimes
\mathrm{Sym}^{a-b}(K^\ell \otimes W^*)
\]
is via the map
\[
\xymatrix{
& {\textstyle\bigwedge^{b}} K^\ell \otimes \mathrm{Sym}^{a-b}(K^\ell
\otimes W^*)\\
{\textstyle\bigwedge^{a}} K^\ell\otimes {\textstyle\bigwedge^{a-b}}W^*
\ar[r]^(.4){\Phi\otimes1} & {\textstyle\bigwedge^{b}} K^\ell \otimes
{\textstyle\bigwedge^{a-b}} 
K^\ell \otimes  {\textstyle\bigwedge^{a-b}}W^* \ar[u]^{1\otimes\Psi}
}
\]
From here, the desired commutivity follows easily.  
\end{proof}

\subsection{Proof of the Main Theorem}
We now have all of the tools needed for our main result.

\begin{proof}[Proof of Theorem~\ref{WlRF}] 
The key idea of the proof (taken from \cite[Thm.\ 1.2]{ES}) is to
replace a coherent sheaf $\Fc$ on $\P(W)$ with its Beilinson monad,
which is the complex of sheaves $B^\bullet(\Fc)$ on $\P(W)$ with
\[
B^p(\Fc) = \bigoplus_i H^i(\Fc(p-i)) \otimes
\Omega_{\P(W)}^{i-p}(i-p),
\]
and differentials coming from the corresponding differentials in the
Tate resolution via the correspondence
\[
\varphi: \ed(a)\otimes H \to \ed(b)\otimes H' \longmapsto
\widetilde{\phi}: \Omega_{\P(W)}^{a}(a)\otimes
H \to \Omega_{\P(W)}^{b}(b)\otimes H'
\]
defined as follows: $\varphi : \ed(a)\otimes H \to \ed(b)\otimes H'$
induces $\phi: H \to \bigwedge^{a-b}W^*\otimes
H'$, and contraction with $\phi$ gives
$\widetilde{\phi}: \Omega_{\P(W)}^{a}(a)\otimes H \to
\Omega_{\P(W)}^{b}(b)\otimes H'$.  Theorem~6.1 of \cite{EFS} is
Beilinson's result that $B^\bullet(\Fc)$ is exact except at $p=0$,
where the homology is $\Fc$.  It follows that
\[
\mathbf{R}p_{1*}(p_2^*\Fc) = \mathbf{R}p_{1*}(p_2^*B^\bullet(\Fc))
\]
in the derived category.  Since $B^\bullet(\Fc)$ consists of direct
sums of sheaves $\Omega_{\P(W)}^{a}(a)$, Corollary~\ref{propcor}
implies that
\[
\mathbf{R}p_{1*}(p_2^*B^\bullet(\Fc)) = p_{1*}(p_2^*B^\bullet(\Fc)),
\]
where $p_{1*}(p_2^*B^p(\Fc))$ is the sheaf associated to the graded
$A_\ell$-module
\[
\bigoplus_i H^i(\Fc(p-i)) \otimes \Wb_\ell(\ed(i-p)) =
\Wb_\ell(T^p(\Fc)). 
\]
Furthermore, Proposition~\ref{homomorphism} easily implies that the
differentials in this complex come from the Tate resolution via the
$\Wb_\ell$ functor.  This completes the proof of the theorem.
\end{proof}

This theorem implies the following uniqueness result for our
generalized Weyman complexes.  

\begin{corollary}
\label{uniqueness}
The complex $\Wb^\bullet_\ell(\Fc)$ is the unique minimal free complex
quasi-isomorphic to $\mathbf{R}p_{1*}(p_2^*B^\bullet(\Fc))$.
\end{corollary}

\begin{proof}
The proof is similar to the proof of \cite[(5.2.5)]{weyman2}.
\end{proof} 

\section{Properties of Generalized Weyman Complexes}

In Section 1, we worked with coherent sheaves on on projective space,
yet Weyman complexes were originally defined for vector bundles on
projective varieties.  Fortunately, the theory of Section 1 adapts
easily to projective varieties.  

\subsection{Generalized Weyman Complexes for Irreducible Projective
Varieties} Let $X$ be an irreducible projective variety with a very
ample line bundle $\Oc_X(1)$.  Let $W = H^0(X,\Oc_X(1))$ and fix an
integer $1 \le \ell \le \dim(W) - 1$.  The incidence variety
\[
Z_\ell = \{(f_1,\dots,f_\ell,x) \in W^{\ell}\times X \mid f_1(x) =
\cdots = f_\ell(x) = 0\}
\]
gives the commutative diagram of projections
\[
\xymatrix{
Z_\ell \ar[dr] \ar@/^/[drr]^{p_2} \ar@/_/[ddr]_{p_1} \\
& W^{\ell} \times X \ar[r]_(.62){\pi_2} \ar[d]^{\pi_1} & X\\
& W^{\ell}.
}
\]
Also let $A_\ell = \mathrm{Sym}((W^{\ell})^*)$ be the coordinate ring
of the affine space $W^{\ell}$.

\begin{theorem}
\label{GWC}
For every coherent sheaf $\Gc$ on $X$, there is a complex
$\Wb_\ell^\bullet(\Gc)$ with the following properties:
\begin{enumerate}
\item $\Wb_\ell^\bullet(\Gc)$ is functorial in $\Gc$.
\item $\Wb_\ell^\bullet(\Gc)$ is a minimal
complex of free graded $A_\ell $-modules.
\item $\Wb_\ell^p(\Gc) = \bigoplus_i {\textstyle\bigwedge^{i-p}} K^{\ell}
\otimes H^i(X,\Gc(p-i)) \otimes A_\ell(p-i)$.
\item When we regard $\Wb_\ell^\bullet(\Gc)$ as a
complex of sheaves on $W^{\ell}$, we have
\[
\Wb_\ell^\bullet(\Gc) \simeq \Rb p_{1*}(p_2^*\Gc). 
\]
\item $\Wb^\bullet(\Gc)$ is the unique minimal free graded complex of
$A_\ell $-modules 
quasi-isomorphic to $\mathbf{R}p_{1*}(p_2^*B^\bullet(\Gc))$.
\item If $\ell > \dim(\mathrm{supp}(\Gc))$, then $\Gc$ is determined
up to isomorphism by $\Wb_\ell^\bullet(\Gc)$.
\item  Let $i: X \hookrightarrow \P(W)$ be the
projective embedding induced by
$\Oc_X(1)$ and write the Tate resolution of $i_*\Gc$ on $\P(W)$ as
\[
T^p(i_*\Gc) = \bigoplus_i \ed(i-p) \otimes H^i(X,\Gc(p-i)).
\]
Then $\Wb_\ell^\bullet(\Gc) = \Wb_\ell(T^\bullet(i_*\Gc))$.  In
particular, the differentials in $\Wb_\ell^\bullet(\Gc)$ are completely
determined by the corresponding differentials in the Tate resolution.
\item When $\Gc$ is a vector bundle, $\Wb_\ell^\bullet(\Gc)$ is
isomorphic to the complex constructed by Weyman in \cite{weyman1,
weyman2, weyman3}.
\end{enumerate}
\end{theorem}

\begin{proof}
This follows from Proposition~\ref{WldeterminesF},
Theorem~\ref{WlRF}, and Corollary~\ref{uniqueness}.  
\end{proof}

\begin{definition}
\label{GWCdef}
$\Wb^\bullet_\ell(\Gc)$ is the \emph{\bfseries $\ell$-th generalized
Weyman complex} of $\Gc$.
\end{definition}

\subsection{The Excluded Case} All of our results assume $\ell
\le \dim(W)-1$.  Here we discuss the complications that arise when
$\ell = \dim(W) = N+1$.

First observe that Proposition~\ref{WOaa} fails when $\ell = \dim(W)$.
To see why, set $\ell = \dim(W) = N+1$ in \eqref{wey_complex}
to obtain
\begin{align*}
\cdots \longrightarrow 0 &\longrightarrow {\textstyle\bigwedge^{N+1}}
K^{N+1} \otimes 
H^N(\P(W),\Omega_{\P(W)}^a(a-N-1))  \otimes A_{N+1}(-N-1)\\
 &\longrightarrow {\textstyle\bigwedge^{a}} K^{N+1} \otimes
H^a(\P(W),\Omega_{\P(W)}^a)  \otimes A_{N+1}(-a) \longrightarrow 0
\longrightarrow \cdots, 
\end{align*}
where the displayed terms are in degrees $q = -1,0$.  Note also that
\[
H^N(\P(W),\Omega_{\P(W)}^a(a-N-1)) \simeq {\textstyle\bigwedge^a} W\quad
\text{and} \quad
H^a(\P(W),\Omega_{\P(W)}^a) \simeq K.
\]
Since $H^q(Z_{N+1},p_2^*\Omega^a_{\P(W)}(a))$ is the $q$-th cohomology
of this complex, it follows that the isomorphism of
Proposition~\ref{WOaa} is replaced with an exact sequence
\begin{align*}
0 \to {\textstyle\bigwedge^{N+1}} K^{N+1}\, &\otimes
{\textstyle\bigwedge^{a}} W \otimes A_{N+1}(-N-1) \to\\
&\Wb_{N+1}(\ed(a)) \to H^0(Z_{N+1},p_2^*\Omega^a_{\P(W)}(a)) \to 0.
\end{align*}
It follows that Corollary~\ref{propcor} needs to be replaced with a
similar exact sequence.  We suspect that the proof of
Theorem~\ref{WlRF} could be adapted to this situation, though the
proof would be considerably more complicated.  

From the point of view of resultants, however, the assumption $\ell
\le \dim(W)-1$ is not overly restrictive.  When one has a vector
bundle on $\P(W)$ and $\ell = \dim(W)$, our theory does not apply, but
one still has the classical Weyman complex from \cite{weyman1,
weyman2, weyman3}.  Since the resultant is just a power of the
determinant in this case, we do not lose much by excluding $\ell =
\dim(W)$.


\subsection{Weyman Complexes and Fourier-Mukai Transform}  Here
we observe that the Weyman complex
$\mathbf{W}_{\ell}^\bullet(\Vc)$ can be regarded as a
Fourier-Mukai transform when $\Vc$ is a vector bundle on an
irreducible projective variety $X$.

We begin with the definition of a Fourier-Mukai transform \cite[Section
5]{huybrechts}.  Denote by $\mathbf{D}^b(A)$ a bounded derived
category of coherent sheaves of a scheme $A$.

\begin{definition}
Let $M$ and $N$ be projective varieties. Given the two projections 
\[
p : M\times N \to M, \quad q : M\times N\to N
\]
and an object $\Pc\in\mathbf{D}^b(M\times N)$, the
\emph{\bfseries Fourier-Mukai transform} is the functor
\[
\mathrm{\Phi}_{\Pc} : \mathbf{D}^b(M)\to \mathbf{D}^b(N),
\quad \Ec^\bullet\mapsto \mathbf{R} q_* (p^*
\Ec^\bullet\otimes^{\mathbf{L}}\Pc).
\]
The object $\Pc$ is called a \emph{\bfseries Fourier-Mukai
kernel} of the Fourier-Mukai transform.
\end{definition}

Notice that if $\Pc$ is a complex of locally free sheaves,
then $\otimes^\mathbf{L}$ is the usual tensor product.  Also, 
$p^*$ is the usual pull-back as the projection map $p$ is flat.

Denote by $W$ the set of global sections of the very ample line bundle
$\Oc_X(1)$ on $X$ and recall the basic diagram of projections
\[
\xymatrix{
Z_\ell \ar[dr] \ar@/^/[drr]^{p_2} \ar@/_/[ddr]_{p_1} \\
& W^{\ell} \times X \ar[r]_(.62){\pi_2} \ar[d]^{\pi_1} & X\\
& W^{\ell},
}
\]
for the 
incidence variety 
\[
Z_\ell = \{(f_1,\dots,f_\ell,x) \in W^{\ell}\times X \mid f_1(x) =
\cdots = f_\ell(x) = 0\}.
\]

\begin{theorem}
Let $X$, $W$ and $Z_\ell$ be as above and assume $1\le \ell \le
\dim(W)- 1$.  If $\Vc$ is a vector bundle on $X$, then the complex
$\Wb^\bullet_\ell(\Vc)$ represents the Fourier-Mukai transform
\[
\mathrm{\Phi}_{\Oc_{Z_{\ell}}}(\Vc) =
\mathbf{R}\pi_{1*}(\pi_2^*\Vc\otimes^\mathbf{L} \Oc_{Z_{\ell}})  
\]
with respect to the kernel $\Oc_{Z_{\ell}}$.
\end{theorem}

\begin{proof}
We have seen that $\mathbf{W}_{\ell}^\bullet(\Vc)$, when regarded as a
complex in $\mathbf{D}^b(W^\ell)$, represents
\[
\mathbf{R} p_{1*}(p_2^*\Vc) \simeq  
\mathbf{R} \pi_{1*}(\pi_2^*\Vc\otimes \Oc_{Z_{\ell}}).
\]
However, $\Vc$ is locally free, which implies that the same is true
for $\pi_2^*\Vc$.  Hence $\Fc \mapsto \pi_2^*\Vc\otimes \Fc$ is an
exact functor, so that $\pi_2^*\Vc\otimes -=
\pi_2^*\Vc\otimes^\mathbf{L} -$.  The theorem follows immediately.
\end{proof}

This theorem implies that we can compute the Fourier-Mukai transform
of $\Vc$ for the kernel $\Oc_{Z_\ell}$ using the Tate resolution of
$i_* \Vc$ on $\P(W)$.


\end{document}